\documentclass[12pt,a4paper]{amsart}
\usepackage[latin1]{inputenc}
\usepackage{latexsym}
\usepackage{color,graphicx,shortvrb}
\usepackage{amsmath, amssymb}
\usepackage{amsfonts}
\usepackage[colorlinks, bookmarks=true]{hyperref}

\newtheorem{theorem}{Theorem}[section]
\newtheorem{lemma}[theorem]{Lemma}

\newtheorem{corollary}[theorem]{Corollary}

\theoremstyle{definition}
\newtheorem{definition}[theorem]{Definition}

\author{J. M. Almira, Kh. F. Abu-Helaiel}
\title{A  $p$-adic Montel theorem and locally polynomial functions}
\begin{document}
\keywords{}

%Ultrametric Banach spaces, Approximation theory, Lethargy theorems, p-adic transcendental numbers.}

\subjclass[2010]{}

% 41A65, A1A25, 11J61, 11J81, 11K60.}

\begin{abstract} 
We prove a version of both Jacobi's and Montel's Theorems for the case of continuous functions defined over the field $\mathbb{Q}_p$ of $p$-adic numbers. In particular, we prove that, if 
\[
\Delta_{h_0}^{m+1}f(x)=0 \ \ \text{ for all } x\in\mathbb{Q}_p,
\]
and $|h_0|_p=p^{-N_0}$ then,  for all $x_0\in \mathbb{Q}_p$, the restriction of $f$  over the set $x_0+p^{N_0}\mathbb{Z}_p$  coincides with a  polynomial $p_{x_0}(x)=a_0(x_0)+a_1(x_0)x+\cdots+a_m(x_0)x^m$. Motivated by this result, we compute the general solution of the functional equation with restrictions given by 
\begin{equation*}
\Delta_h^{m+1}f(x)=0 \ \ (x\in X \text{ and } h\in B_X(r)=\{x\in X:\|x\|\leq r\}),
\end{equation*}
whenever $f:X\to Y$, $X$ is an ultrametric normed  space over a non-Archimedean valued field $(\mathbb{K},|\cdot|)$ of characteristic zero, and $Y$ is a $\mathbb{Q}$-vector space. By obvious reasons, we call these functions \textit{uniformly locally polynomial}. 

 \end{abstract}

\maketitle

\markboth{J. M. Almira, Kh. F. Abu-Helaiel}{ $p$-adic Montel Theorem and locally polynomial functions}
%\section{Introduction}

\section{Motivation}
Given a commutative group $(G,+)$, a nonempty set $Y$, and a function $f:G\to Y$, we consider the set of periods of $f$, $\mathfrak{P}_0(f)=\{g\in G:f(w+g)=f(w)\text{ for all } w\in G\}$. Obviously, $\mathfrak{P}_0(f)$ is always a subgroup of $G$ and, in some special cases, these groups are well known and, indeed, have a nice structure.  For example, a famous result proved by Jacobi in 1834 claims that  if $f:\mathbb{C}\to\widehat{\mathbb{C}}$ is a non constant meromorphic function defined on the complex numbers, then  $\mathfrak{P}_0(f)$ is a discrete subgroup of $(\mathbb{C},+)$. This reduces the possibilities to the following three cases: $\mathfrak{P}_0(f)=\{0\}$, or $\mathfrak{P}_0(f)=\{nw_1:n\in\mathbb{Z}\}$ for a certain complex number $w_1\neq 0$, or $\mathfrak{P}_0(f)=\{n_1w_1+n_2w_2:(n_1,n_2)\in \mathbb{Z}^2\}$ for certain complex numbers $w_1,w_2$ satisfying $w_1w_2\neq 0$ and $w_1/w_2\not\in\mathbb{R}$. In particular, these functions cannot have three independent periods and there exist functions  $f:\mathbb{C}\to\mathbb{C}$  with two independent periods $w_1,w_2$ as soon as $w_1/w_2\not\in\mathbb{R}$. These functions are called doubly periodic (or elliptic) and have an important role in complex function theory \cite{JS}.  Analogously, if the function $f:\mathbb{R}\to\mathbb{R}$ is continuous and non constant, it does not admit  two $\mathbb{Q}$-linearly independent periods. 

These results can be formulated in terms of functional equations since $h$ is a period of $f:G\to Y$ if and only if $f$ solves the functional equation
$\Delta_hf(x)=0 \ \ (x\in G).$ Thus, Jacobi's theorem can be formulated as a result which characterizes the constant functions as those meromorphic  functions $f:\mathbb{C}\to\widehat{\mathbb{C}}$  which solve a system of functional equations of the form 
\begin{equation}\label{JC}
\Delta_{h_1}f(z)=\Delta_{h_2}f(z)=\Delta_{h_3}f(z)=0 \ \ (z\in \mathbb{C})
\end{equation}
for three independent periods $\{h_1,h_2,h_3\}$ (i.e., $h_3\not\in h_1\mathbb{Z}+h_2\mathbb{Z}$). For the real case, the result states that, if $\dim_{\mathbb{Q}}\mathbf{span}_{\mathbb{Q}}\{h_1,h_2\}=2$, the continuous function $f:\mathbb{R}\to\mathbb{R}$ is a constant function if and only if it solves the system of functional equations 
\begin{equation} \label{JR}
\Delta_{h_1}f(x)=\Delta_{h_2}f(x)= 0\ (x\in \mathbb{R}).
\end{equation}
In 1937  Montel \cite{montel} proved an interesting nontrivial generalization of Jacobi's theorem. Concretely, he substituted in the equations $(\ref{JC}),(\ref{JR})$ above the first difference operator $\Delta_h$ by  the higher differences operator $\Delta^{m+1}_h$ (which is inductively defined by $\Delta_h^{n+1}f(x)=\Delta_h(\Delta_h^nf)(x)$, $n=1,2,\cdots$) and proved that these equations are appropriate for the characterization of ordinary polynomials. Concretely, he proved the following result:
\begin{theorem}[Montel] Assume that $f:\mathbb{C}\to\mathbb{C}$ is an analytic function  which solves a system of functional equations of the form 
\begin{equation}\label{JC}
\Delta_{h_1}^{m+1}f(z)=\Delta_{h_2}^{m+1}f(z)=\Delta_{h_3}^{m+1}f(z)=0 \ \ (z\in \mathbb{C})
\end{equation}
for three independent periods $\{h_1,h_2,h_3\}$. Then $f(z)=a_0+a_1z+\cdots+a_mz^m$ is an ordinary polynomial with complex coefficients and degree $\leq m$. Furthermore, if $\{h_1,h_2\}\subset \mathbb{R}$ satisfy 
$\dim_{\mathbb{Q}}\mathbf{span}_{\mathbb{Q}}\{h_1,h_2\}=2$, the continuous function $f:\mathbb{R}\to\mathbb{R}$ is an ordinary polynomial with real coefficients and degree $\leq m$ (i.e., $f(x)=a_0+a_1x+\cdots+a_mx^m$) if and only if it solves the system of functional equations 
\begin{equation} \label{JR}
\Delta_{h_1}^{m+1}f(x)=\Delta_{h_2}^{m+1}f(x)= 0\ (x\in \mathbb{R}).
\end{equation}
\end{theorem}

The functional equation $\Delta_h^{m+1}f(x)=0$ had already been introduced in the literature by M. Fr\'{e}chet in 1909 as a particular case of the functional equation  
%Concretely, in his seminal paper \cite{frechet}, he proved that the 
%Let $X, Y$ be two $\mathbb{Q}$-vector spaces and let $f:X\to Y$. We say that $f$ satisfies Fr\'{e}chet's functional equation of order $s-1$ if 
\begin{equation}\label{fre}
\Delta_{h_1h_2\cdots h_{m+1}}f(x)=0 \ \ (x,h_1,h_2,\dots,h_{m+1}\in \mathbb{R}),
\end{equation}
where $f:\mathbb{R}\to\mathbb{R}$ and $\Delta_{h_1h_2\cdots h_s}f(x)=\Delta_{h_1}\left(\Delta_{h_2\cdots h_s}f\right)(x)$, $s=2,3,\cdots$. In particular, after Fr\'{e}chet's 
seminal paper \cite{frechet}, the solutions of \eqref{fre} are named ``polynomials'' by the Functional Equations community, since it is known that, under very mild regularity conditions on $f$, if $f:\mathbb{R}\to\mathbb{R}$ satisfies \eqref{fre}, then $f(x)=a_0+a_1x+\cdots a_{s-1}x^{s-1}$ for all $x\in\mathbb{R}$ and certain constants $a_i\in\mathbb{R}$. For example, in order to have this property, it is enough for $f$ being locally bounded \cite{frechet}, \cite{almira_antonio}, but there are stronger results \cite{ger1}, \cite{kuczma1}, \cite{mckiernan}. The equation \eqref{fre} can be studied for functions $f:X\to Y$  whenever $X, Y$ are two  $\mathbb{Q}$-vector spaces and the variables $x,h_1,\cdots,h_{m+1}$ are assumed to be elements of $X$:
\begin{equation}\label{fregeneral}
\Delta_{h_1h_2\cdots h_{m+1}}f(x)=0 \ \ (x,h_1,h_2,\dots,h_{m+1}\in X).
\end{equation}
In this context, the general solutions of \eqref{fregeneral} are characterized as functions of the form $f(x)=A_0+A_1(x)+\cdots+A_n(x)$, where $A_0$ is a constant and $A_k(x)=A^k(x,x,\cdots,x)$ for a certain $k$-additive symmetric function $A^k:X^k\to Y$ (we say that $A_k$ is the diagonalization of $A^k$). In particular, if $x\in X$ and $r\in\mathbb{Q}$, then $f(rx)=A_0+rA_1(x)+\cdots+r^nA_n(x)$. Furthermore, it is known that $f:X\to Y$ satisfies \eqref{fregeneral} if and only if it satisfies 
\begin{equation}\label{frepasofijo}
\Delta_{h}^{m+1}f(x):=\sum_{k=0}^{m+1}\binom{m+1}{k}(-1)^{s-k}f(x+kh)=0 \ \ (x,h\in X).
\end{equation}
A proof of this fact follows directly from Djokovi\'{c}'s Theorem \cite{Dj} (see also \cite[Theorem 7.5, page 160]{HIR}, \cite[Theorem 15.1.2., page 418]{kuczma}), which states that  the operators $\Delta_{h_1 h_2\cdots h_s}$ satisfy the equation
\begin{equation}\label{igualdad}
\Delta_{h_1\cdots h_s}f(x)=
\sum_{\epsilon_1,\dots,\epsilon_s=0}^1(-1)^{\epsilon_1+\cdots+\epsilon_s}
\Delta_{\alpha_{(\epsilon_1,\dots,\epsilon_s)}(h_1,\cdots,h_s)}^sf(x+\beta_{(\epsilon_1,\dots,\epsilon_s)}(h_1,\cdots,h_s)),
\end{equation}
where $\alpha_{(\epsilon_1,\dots,\epsilon_s)}(h_1,\cdots,h_s)=(-1)\sum_{r=1}^s\frac{\epsilon_rh_r}{r}$ and $\beta_{(\epsilon_1,\dots,\epsilon_s)}(h_1,\cdots,h_s)=\sum_{r=1}^s\epsilon_rh_r$.  

%Given $D\subseteq X$, the function $f:D\to Y$ is named a ``polynomial on $D$'' if $f$ satisfies \eqref{frepasofijo} for a certain $m\in\mathbb{N}$ and all $x,h\in X$ such that $\{x,x+h,\cdots,x+(m+1)h\}\subseteq D$. A natural %problem, that has been solved by R. Ger \cite{ger}, is to study the conditions under which a polynomial $f$ on $D$ can be extended (and how to make this) to a polynomial on $X$. 

In section 2 of this paper we prove a version of both Jacobi's and Montel's Theorems for the case of continuous functions defined over the field $\mathbb{Q}_p$ of $p$-adic numbers. In particular, we prove that, if 
\[
\Delta_{h_0}^{m+1}f(x)=0 \ \ \text{ for all } x\in\mathbb{Q}_p,
\]
and $|h_0|_p=p^{-N_0}$ then,  for all $x_0\in \mathbb{Q}_p$, the restriction of $f$  over the set $x_0+p^{N_0}\mathbb{Z}_p$  coincides with a  polynomial $p_{x_0}(x)=a_0(x_0)+a_1(x_0)x+\cdots+a_m(x_0)x^m$. Motivated by this result, we compute, in the last section of this paper, the general solution of the functional equation with restrictions given by 
\begin{equation}\label{LocPol}
\Delta_h^{m+1}f(x)=0 \ \ (x\in X \text{ and } h\in B_X(r)=\{x\in X:\|x\|\leq r\}),
\end{equation}
whenever $f:X\to Y$, $X$ is an ultrametric normed  space over a non-Archimedean valued field $(\mathbb{K},|\cdot|)$ of characteristic zero (so that, it contains a copy of $\mathbb{Q}$), and $Y$ is a $\mathbb{Q}$-vector space. By obvious reasons, we call these functions \textit{uniformly locally polynomial}. 

The definition and basic properties of $\mathbb{Q}_p$ and ultrametric normed spaces over non-Archimedean valued fields can be found, for example, in \cite{gouvea}, \cite{perez} and \cite{robert}. In any case, we would like to stand up the fact  that, if  $X$ is an ultrametric normed  space over a non-Archimedean valued field $(\mathbb{K},|\cdot|)$ and $x,y\in X$ satisfy $\|x\|>\|y\|$, then $\|x+y\|=\|x\|$ (see, e.g.,  \cite[page 22]{perez}).  

%over a non Archimedian field $\mathbb{K}$. 

\section{$p$-adic Montel's Theorem}

\begin{theorem}\label{teo1} Let $Y$ be a topological space with infinitely many points, and let $N\in\mathbb{Z}$. Then there are continuous functions $f:\mathbb{Q}_p\to Y$ such that
\[
\Delta_hf(x)=0 \Leftrightarrow h\in p^N\mathbb{Z}_p.
\]
These functions are obviously non-constant.
\end{theorem}
\noindent \textbf{Proof. } We know that   $p^N\mathbb{Z}_p$ is an additive subgroup of $\mathbb{Q}_p$. Moreover, the quotient group $\mathbb{Q}_p/p^N\mathbb{Z}_p$ is isomorphic to the well known Pr\"{u}fer group $C_{p^{(\infty)}}=\bigcup_{k=0}^\infty C_{p^k}$ (here, $C_{p^k}$ denotes the cyclic group of order $p^k$). In particular, there exists an infinite countable set $S_N\subset \mathbb{Q}_p$ such that $\{s+p^N\mathbb{Z}_p\}_{s\in S_N}$ defines a partition of $\mathbb{Q}_p$ in clopen sets. If $\lambda:S_N\to Y$ is any inyective map, the function $f:\mathbb{Q}_p\to Y$ defined by $f(x)=\lambda(s)$ if and only if $x\in s+p^N\mathbb{Z}_p$, $s\in S_N$, satisfies our requirements. {\hfill $\Box$}

%Obviously, $\Delta_hf(x)=0$ if and only if $h$ is a period of $f$. Let us denote by $\mathfrak{P}_0(f)$ the set of periods of the function $f$. $\mathfrak{P}_0(f)$ is an additive subgroup of $\mathbb{Q}_p$. What is more, the %following result holds:

\begin{lemma}
Let $(Y,d)$ be a metric space. If $f:\mathbb{Q}_p\to Y$ is continuous and $h\in \mathfrak{P}_0(f)$, $|h|_p=p^{-N}$,  then $p^N\mathbb{Z}_p\subseteq \mathfrak{P}_0(f)$. In particular, $\mathfrak{P}_0(f)$ is a clopen additive subgroup of $\mathbb{Q}_p$.
\end{lemma}

\noindent \textbf{Proof. } The continuity of $f$ implies that $\mathfrak{P}_0(f)$ is closed. Let us include, for the sake of completeness, the proof of this fact.  Let $\{h_k\}\subset \mathfrak{P}_0(f)$, $\lim_{k\to\infty}h_k=h$. Then
\begin{eqnarray*}
0 &\leq& d(f(x+h),f(x))\leq d(f(x+h),f(x+h_k))+d(f(x+h_k),f(x))\\
&=& d(f(x+h),f(x+h_k))\to 0  \ \ (\text{ for } k\to\infty),
\end{eqnarray*}
Hence $f(x+h)=f(x)$ for all $x\in\mathbb{Q}_p$. Thus $h\in \mathfrak{P}_0(f)$.

Take $h\in \mathfrak{P}_0(f)$, $|h|_p=p^{-N}$. Then $\overline{\{kh\}_{k=1}^\infty}^{\mathbb{Q}_p}=p^N\mathbb{Z}_p\subset \mathfrak{P}_0(f)$. This ends the proof. {\hfill $\Box$}

\begin{corollary}[p-adic version of Jacobi's Theorem] Let $(Y,d)$ be a metric space. If $f:\mathbb{Q}_p\to Y$ is continuous and non-constant, then $\mathfrak{P}_0(f)=\{0\}$ or $\mathfrak{P}_0(f)=p^N\mathbb{Z}_p$ for a certain $N\in\mathbb{Z}$. In particular, the continuous function $f:\mathbb{Q}_p\to Y$ is a constant if and only if it contains an unbounded sequence of periods.
\end{corollary}

\noindent \textbf{Proof. } It is well known (and easy to prove) that every proper nontrivial closed additive subgroup of $\mathbb{Q}_p$ is of the form  $p^N\mathbb{Z}_p$ for a certain $N\in\mathbb{Z}$ (see, e.g., \cite[p. 283, Proposition 52.3]{classicalfields}).   {\hfill $\Box$}

 \begin{theorem}[p-adic version of Montel's Theorem]  \label{mont} Let $(\mathbb{K},|\cdot|_{\mathbb{K}})$ be a valued field such that $\mathbb{Q}_p\subseteq \mathbb{K}$ and the inclusion $\mathbb{Q}_p \hookrightarrow\mathbb{K}$ is continuous. Let us assume that $f:\mathbb{Q}_p\to \mathbb{K}$ is continuous, and define
 \[
 \mathfrak{P}_m(f)=\{h\in\mathbb{Q}_p:\Delta_h^{m+1}f=0\}.
 \]
Then either $\mathfrak{P}_m(f) =\{0\}$,  $\mathfrak{P}_m(f) = \mathbb{Q}_p$, or $\mathfrak{P}_m(f)=p^N\mathbb{Z}_p$ for a certain  $N\in\mathbb{Z}$. Furthermore, all these cases are effectively attained by some appropriate instances of the function $f$. Finally, for all $a\in\mathbb{Q}_p$ there exists constants $a_0,a_1,\cdots,a_m\in \mathbb{K}$ such that $f(x)=a_0+\cdots+a_mx^m$ for all $x\in a+\mathfrak{P}_m(f)$.
In particular, $f$ is a polynomial of degree $\leq m$ if and only if $\mathfrak{P}_m(f)$ contains an unbounded sequence. \end{theorem}

\noindent \textbf{Proof. }  Assume  $\mathfrak{P}_m(f) \neq \{0\}$. Let  $h_0\in \mathfrak{P}_m(f)$, $h_0\neq 0$. Then
$\Delta_{h_0}^{m+1}f(x)=0$ for all $x\in\mathbb{Q}_p$. Let $x_0\in\mathbb{Q}_p$ and let $p_0(t)\in\mathbb{K}[t]$ be the polynomial of degree $\leq m$ such that $f(x_0+kh_0)=p_0(x_0+kh_0)$ for all $k\in\{0,1,\cdots,m\}$ (this polynomial exists and it is unique, thanks to Lagrange's interpolation formula). Then
\begin{eqnarray*}
0 &=&
\Delta_{h_0}^{m+1}f(x_0)=\sum_{k=0}^m\binom{m+1}{k}(-1)^{m+1-k}f(x_0+kh_0)+f(x_0+(m+1)h_0)\\
&=& \sum_{k=0}^m\binom{m+1}{k}(-1)^{m+1-k}p_0(x_0+kh_0)+f(x_0+(m+1)h_0)\\
&=& -p_0(x_0+(m+1)h_0)+f(x_0+(m+1)h_0),
\end{eqnarray*}
since  $0=\Delta_{h_0}^{m+1}p(x_0)=\sum_{k=0}^{m+1}\binom{m+1}{k}(-1)^{m+1-k}p_0(x_0+kh_0)$. This means that $f(x_0+(m+1)h_0)=p_0(x_0+(m+1)h_0)$. In particular, $p_0=q$, where $q$ denotes the polynomial of degree $\leq m$ which interpolates $f$ at the nodes $\{x_0+kh_0\}_{k=1}^{m+1}$. This argument can be repeated to prove that $p_0$ interpolates $f$ at all the nodes $x_0+h_0\mathbb{N}$. On the other hand, if $|h_0|_p=p^{-N}$, then $h_0\mathbb{N}$ is a dense subset of $p^N\mathbb{Z}_p$. It follows that $f_{|x_0+p^N\mathbb{Z}_p}=(p_0)_{|x_0+p^N\mathbb{Z}_p}$, since $f$ is continuous. Thus, we have proved that the restrictions of $f$ over the sets of the form $x_0+p^N\mathbb{Z}_p$ are polynomials of degree $\leq m$. On the other hand, we have already shown the  existence an infinite countable set $S_N\subset \mathbb{Q}_p$ such that $\{s+p^N\mathbb{Z}_p\}_{s\in S_N}$ is a partition of $\mathbb{Q}_p$ in clopen sets. Hence there exists a family of polynomials $\{p_s(t)\}_{s\in S_N}\subset \mathbb{K}[t]$ such that $\deg p_s\leq m$ for all $s\in S_N$ and $f(x)=p_s(x)$ if and only if $x\in s+p^N\mathbb{Z}_p$, $s\in S_N$. Let $h\in p^N\mathbb{Z}_p$. We want to show that $h\in \mathfrak{P}_m(f)$. Now, given $x\in \mathbb{Q}_p$, there exists $s\in S_N$ such that $x+p^N\mathbb{Z}_p=s+p^N\mathbb{Z}_p$. In particular, $f_{|\{x,x+h,x+2h,\cdots,x+mh,x+(m+1)h\}}=(p_s)_{|\{x,x+h,x+2h,\cdots,x+mh,x+(m+1)h\}}$, so that $\Delta_h^{m+1} f(x)= \Delta_h^{m+1} p_s(x)=0$. This proves that $p^N\mathbb{Z}_p\subseteq \mathfrak{P}_m(f)$.

We may summarize the the arguments above by claiming that if $h_0\in \mathfrak{P}_m(f)$ and $|h_0|_p=p^{-N}$, then $p^N\mathbb{Z}_p\subseteq \mathfrak{P}_m(f)$ and
\begin{equation}\label{fam}
f(x)=p_s(x) \Leftrightarrow x\in s+p^N\mathbb{Z}_p \text{ and } s\in S_N,
\end{equation}
where  $\{p_s(t)\}_{s\in S_N}$ is a family of polynomials $p_s\in \mathbb{K}[t]$ verifying  $\deg p_s\leq m$ for all $s\in S_N$, and  $\{s+p^N\mathbb{Z}_p\}_{s\in S_N}$ is a partition of $\mathbb{Q}_p$. Furthermore, for any function $f$ satisfying $(\ref{fam})$, we have that  $p^N\mathbb{Z}_p\subseteq \mathfrak{P}_m(f)$.

Thus, there are just two possibilities we may consider:
\begin{itemize}
\item[Case 1:]  $\inf\{N\in\mathbb{Z}:p^N\mathbb{Z}_p\subseteq \mathfrak{P}_m(f)=\}-\infty$.  
\end{itemize}
\noindent In this case $\mathfrak{P}_m(f)=\mathbb{Q}_p$ and $f$ is a polynomial of degree $\leq m$.
\begin{itemize}
\item[Case 2: ] $\inf\{N\in\mathbb{Z}:p^N\mathbb{Z}_p\subseteq \mathfrak{P}_m(f)\}=N_0$.  \end{itemize}

\noindent In this case
 $\mathfrak{P}_m(f)=p^{N_0}\mathbb{Z}_p$ and $f$ satisfies $(\ref{fam})$ with $N=N_0\in\mathbb{Z}$.

This ends the proof. {\hfill $\Box$}

\begin{definition} Given $f:\mathbb{Q}_p\to K$ a continuous function, we say that $f$ is locally an ordinary polynomial if for each $x_0\in \mathbb{Q}_p$ there exist a positive radius $r>0$ and constants $a_0,a_1,\cdots,a_m\in \mathbb{K}$ such that $f(x)=a_0+\cdots+a_mx^m$ for all $x\in x_0+B_{\mathbb{Q}_p}(r)$. We say that $f$ is uniformly locally an ordinary polynomial if, furthermore,  the radius $r>0$ can be chosen the same for all $x_0\in\mathbb{Q}_p$. 
\end{definition}

\begin{corollary}\label{cor_montel} If $f:\mathbb{Q}_p\to K$ is continuous, then $f$ is uniformly locally an ordinary polynomial if and only if  $\mathfrak{P}_m(f)\neq \{0\}$. There are locally ordinary polynomial functions   $f:\mathbb{Q}_p\to K$ such that  $\mathfrak{P}_m(f)= \{0\}$.
\end{corollary}

\begin{proof} The first claim is an easy consequence of Theorem \ref{mont}. To prove the existence of locally ordinary polynomials which are not uniformly locally ordinary polynomials it will be enough to construct an example. With this objective in mind, we define $f:\mathbb{Q}_p\to\mathbb{Q}_p$ as follows:
 \[
f(x)=\left\{
\begin{array}{cccccc}
p^n x^m &  & \text{if} & n\in\mathbb{N}  \text{ and }  x\in p^{-n}+p^n\mathbb{Z}_p\\
0 &  & \text{if} & x\not\in  \bigcup_{n=0}^{\infty} (p^{-n}+p^n\mathbb{Z}_p)\\
\end{array},
\right.
\]
Obviously, $f(x)$ is locally an ordinary polynomial of degree $\leq m$. On the other hand, if $N\geq 1$ is a natural number, $h\in\mathbb{Q}_p$, $|h|=p^{-N}$, then $p^{-N(m+1)}+kh\not\in  \bigcup_{n=0}^{\infty} (p^{-n}+p^n\mathbb{Z}_p)$, $k=1,2,\cdots,m+1$. Hence  
\begin{eqnarray*}
\Delta_h^{m+1}f(p^{-N(m+1)}) &= & \sum_{k=0}^{m+1} \binom{m+1}{k}(-1)^{m+1-k}f(p^{-N(m+1)}+kh) \\
&=&  (-1)^{m+1}f(p^{-N(m+1)}) =p^{N(m+1)} p^{-Nm(m+1)} \neq 0 
\end{eqnarray*}
and $\mathfrak{P}_m(f)= \{0\}$.
\end{proof}

%\textcolor{blue}{Hay que tener cuidado con bolas cada vez m\'{a}s chicas!}

\section{Characterization of uniformly locally polynomial functions}

The results of the section above and, in particular, the p-adic Montel's Theorem and Corollary \ref{cor_montel}, motivate us to study, for functions $f:X\to Y$ (where
 $X$ is an ultrametric normed  space over a non-Archimedean valued field $(\mathbb{K},|\cdot|)$ of characteristic zero, and $Y$ is a $\mathbb{Q}$-vector space), the functional equation with restrictions
\begin{equation} \label{fund}
\Delta_h^{m+1}f(x)=0 \ \ (x\in X, h\in B_X(r)=\{x:\|x\|\leq r\}).
\end{equation}  

\begin{definition} We say that $f:X\to Y$ is an uniformly locally polynomial function if it solves the functional equation \eqref{fund} for a certain $r>0$. 
\end{definition}
The best motivation for the concept above should be found in the statement of the following theorem, which is the main result of this section:
\begin{theorem}[Characterization of uniformly locally polynomial functions] \label{CL} Assume that $f:X\to Y$ satisfies \eqref{fund}
%\begin{equation} 
%\Delta_h^{m+1}f(x)=0 \ \ (x\in X, h\in B_X(r)=\{x:\|x\|\leq r\}).
%\end{equation}
and let $\phi(r,m)=r\left(\prod_{k=2}^{m+1} \max\{|1/t|:t=1,2,\cdots,k\}\right)^{-1}$. Then for all $x_0\in X$ there exists a constant $A_{0,x_0}$ and  $k$-additive symmetric  maps 
$$A^{k,x_0}:B_X(\phi(r,m))\times\cdots ^{(k\text{ times})}\times B_X(\phi(r,m))\to Y$$ for $k=1,2,\cdots,m$, such that 
\[
f(x_0+z)=A_{0,x_0}+\sum_{k=1}^mA_{k,x_0}(z) \  \ \text{ for all } \ z\in B_X(\phi(r,m));
\]  
where $A_{k,x_0}(z)=A^{k,x_0}(z,z,\cdots,z)$ is the diagonalization of $A^{k}(z_1,\cdots,z_k)$, $k=1,\cdots,m$. 
\end{theorem}

\begin{lemma} \label{L} Assume that $f:X\to Y$ satisfies \eqref{fund}
%\begin{equation} \label{eculema}
%\Delta_h^{m+1}f(x)=0 \ \ (x\in X, h\in B_X(r)=\{x:\|x\|\leq r\}).
%\end{equation}
and let $$\phi(r,m)=r\left(\prod_{k=2}^{m+1} \max\{|1/t|:t=1,2,\cdots,k\}\right)^{-1}.$$ Then there exist a constant $A_0$ and  $k$-additive symmetric  maps $$A^{k}:B_X(\phi(r,m))\times\cdots ^{(k\text{ times})}\times B_X(\phi(r,m))\to Y$$ for $k=1,2,\cdots,m$, such that 
\[
f(z)=A_0+\sum_{k=1}^mA_{k}(z) \  \ \text{ for all } \ z\in B_X(\phi(r,m));
\]  
where $A_{k}(z)=A^{k}(z,z,\cdots,z)$ is the diagonalization of $A^{k}(z_1,\cdots,z_k)$, $k=1,\cdots,m$. 
\end{lemma} 

\begin{proof}  Assume that $f:X\to Y$ satisfies $(\ref{fund})$, and consider the function $A^m(x_1,\cdots,x_m)=\frac{1}{m!}\Delta_{x_1x_2\cdots x_m}f(0)$. Then $A^m$ is symmetric since the operators $\Delta_{x_i}$, $\Delta_{x_j}$ commute. Furthermore, the identity $\Delta_{x+y}=\Delta_x+\Delta_y+\Delta_{xy}$ implies that 
\begin{eqnarray*}
&\ &A^m(x_1,\cdots,x_{k-1},x+y,x_{k+1},\cdots,x_m) \\
&\ & \ \ \ -A^m(x_1,\cdots,x_{k-1},x,x_{k+1},\cdots,x_m)-A^m(x_1,\cdots,x_{k-1},y,x_{k+1},\cdots,x_m)\\
&\ & = \frac{1}{m!}\left(\Delta_{x_1\cdots x_{k-1}x_{k+1}\cdots x_m}(\Delta_{x+y}-\Delta_x-\Delta_y)f(0) \right)\\
&\ & = \frac{1}{m!}\left(\Delta_{x_1 \cdots x_{k-1}x_{k+1}\cdots x_mxy}f(0) \right).
\end{eqnarray*}  
If we apply Djokovi\'{c}'s theorem to the operator $\Delta_{z_1z_2\cdots z_{m+1}}$, we conclude that, if $z_1,\cdots,z_{m+1}\in B_X(r/\max\{|1/t|:1\leq t\leq m+1\})$, then 
\[
\|\alpha_{(\epsilon_1,\cdots,\epsilon_{m+1})}(z_1,\cdots,z_{m+1})\|=\|(-1)\sum_{t=1}^{m+1}\frac{\epsilon_t}{t}z_t\|\leq \max\{|1/t|:1\leq t\leq m+1\}\max_{1\leq t\leq m+1}\|z_t\|\leq r,
\]
so that $\Delta_{z_1z_2\cdots z_{m+1}}f(x)=0$ for all $x\in X$. In particular, the application $A^m$ is $m$-additive on $B_{X}(r/\max\{|1/t|:1\leq t\leq m+1\})$ and, hence, on all its additive subgroups. In particular, it is $m$-additive on the balls $B_X(\rho)$ for all $\rho \leq r/\max\{|1/t|:1\leq t\leq m+1\}$. 

Let us define the function
 \[
f_1(x)=\left\{
\begin{array}{cccccc}
f(x)-A_m(x) &  & \text{if} &   x\in B_{X}(r/\max\{|1/t|:1\leq t\leq m+1\}) \\
0 &  & \text{if} & x\not\in B_{X}(r/\max\{|1/t|:1\leq t\leq m+1\}) \\
\end{array},
\right.
\]
where $A_m(x)=A^m(x,x,\cdots,x)$ is the diagonalization of $A^m$, and let us compute $\Delta_h^mf_1(x)$ for $x\in X$ and $h\in B_{X}(r/\max\{|1/t|:1\leq t\leq m+1\})$. We divide this computation into two steps:
\begin{itemize}
\item[Step 1:] Assume $x\in B_{X}(r/\max\{|1/t|:1\leq t\leq m+1\})$. 
\end{itemize}
In this case, $x+kh\in  B_{X}(r/\max\{|1/t|:1\leq t\leq m+1\})$ for all $h\in  B_{X}(r/\max\{|1/t|:1\leq t\leq m+1\})$ and all $k\in\mathbb{N}$, so that $\Delta_h^mf_1(x)=\Delta_h^mf(x)-\Delta_h^mA_m(x)$. We compute separately each summand of the second member of this identity. Obviously,
\[
0=\Delta_{hh\cdots hx}f(0)=\Delta^m_h\Delta_xf(0)=\Delta^m_hf(x)-\Delta^m_hf(0),
\] 
since $x,h\in B_{X}(r/\max\{|1/t|:1\leq t\leq m+1\})$. This means that $\Delta^m_hf(x)=\Delta^m_hf(0)=m!A^m(h,\cdots,h)$. On the other hand, a direct computation shows that 
$\Delta_h^mA_m(x)=m!A^m(h,\cdots,h)$, which proves that 
\[
\Delta_h^mf_1(x)=0 \text{ for all } x,h\in B_{X}(r/\max\{|1/t|:1\leq t\leq m+1\}).
\]
\begin{itemize}
\item[Step 2:] Assume $x\not\in B_{X}(r/\max\{|1/t|:1\leq t\leq m+1\})$.
\end{itemize}
Obviously, $\|x\|>\|h\|\geq \|kh\|$ for all $h\in B_{X}(r/\max\{|1/t|:1\leq t\leq m+1\})$ and $k\in\mathbb{N}$, so that 
$\|x+kh\|=\|x\|$ and $\{x+kh\}_{k=0}^m\subset X\setminus B_{X}(r/\max\{|1/t|:1\leq t\leq m+1\})$. Hence $\Delta_h^mf_1(x)=0 $ also in this case. 

Thus, we have proved that  
\begin{equation*} 
f(x)=f_1(x)+A_m(x) \ \ \text{ for all } x\in B_{X}(r/\max\{|1/t|:1\leq t\leq m+1\}),
\end{equation*}
and 
\begin{equation*} 
\Delta_h^{m}f_1(x)=0 \ \ \text{ for all } x\in X \text{ and } h\in B_{X}(r/\max\{|1/t|:1\leq t\leq m+1\}).
\end{equation*}
A repetition of the same arguments will show that $f_1(x)$ admits a decomposition $f_1(x)=f_2(x)+A_{m-1}(x)$ on the ball   
$ B_{X}(\frac{r}{\max\{|1/t|:1\leq t\leq m+1\}\max\{|1/t|:1\leq t\leq m\}})$, with $A_{m-1}$ the diagonalization of an $(m-1)$-additive symmetric function 
\[A^{m-1}:B_{X}(\frac{r}{\max\{|1/t|:1\leq t\leq m+1\}\max\{|1/t|:1\leq t\leq m\}})^{m-1}\to Y,
\]  
and $f_2$ satisfying 
\begin{equation*} 
\Delta_h^{m}f_2(x)=0 \ \ \text{ for all } x\in X \text{ and } h\in B_{X}(\frac{r}{\max\{|1/t|:1\leq t\leq m+1\}\max\{|1/t|:1\leq t\leq m\}}).
\end{equation*}
The iteration of this process leads to a decomposition 
\[
f(x)=f_{m}(x)+A_1(x)+A_2(x)+\cdots+A_m(x), \ \ \text{ for all } x\in B_X(\phi(r,m)), 
\]
with  $A_{k}(z)=A^{k}(z,z,\cdots,z)$ being the diagonalization of the 
$k$-additive symmetric  map 
$A^{k}:B_X(\phi(r,m))\times\cdots ^{(k\text{ times})}\times B_X(\phi(r,m))\to Y$, $k=1,2,\cdots,m$; and 
$\Delta^1_hf_m(x)=0$ for all $x\in X$ and $h\in   B_X(\phi(r,m))$. In particular, this last formula implies that, for $x\in  B_X(\phi(r,m))$, $f_m(x)=f_m(0)=A_0$ is a constant. 
 \end{proof} 

\begin{proof}[Proof of Theorem \ref{CL}] Let us define, for $x_0\in X$, the function $g(x)=f(x_0+x)$. Then $g=\tau_{x_0}(f)$, where $\tau_{x_0}(f)(x)=f(x_0+x)$ is a translation operator. Obviously, the operators $\tau_{x_0}$ and $\Delta_h$ commute, so that
\[
\Delta_h^{m+1}g(x)=\Delta_h^{m+1}\tau_{x_0}f(x)=\tau_{x_0}\Delta_h^{m+1}f(x)=0 \ \ (x\in X, h\in B_X(r)=\{x:\|x\|\leq r\}).
\]
Hence, we can use Lemma \ref{L} with $g$ to conclude that there exist a constant $A_{0,x_0}$ and  $k$-additive maps $A^{k,x_0}:B_X(\phi(r,m))\times\cdots^{(k \text{ times})}\times B_X(\phi(r,m))\to Y$, $k=1,2,\cdots,m$, such that 
\[
f(x_0+z)=A_{0,x_0}+\sum_{k=1}^mA_{k,x_0}(z) \  \ \text{ for all } \ z\in B_X(r);
\]  
where $A_{k,x_0}(z)=A^{k,x_0}(z,z,\cdots,z)$ is the diagonalization of $A^{k,x_0}(z_1,\cdots,z_k)$, $k=0,1,\cdots,m$.
\end{proof}

%\textcolor{blue}{Question: What about the general solution of $\Delta_h^{m+1}f(x)=0$ for all $x\in\mathbb{Q}_p$ and $h\in p^N\mathbb{Z}_p$? And what about the stability question? Frechet %equation does not characterizes ordinary polynomials for the cases of finite extensions of $\mathbb{Q}_p$ and, also, $\mathbb{C}_p$. Does there exists an appropriate equation?}

 \bibliographystyle{amsplain}

\begin{thebibliography}{38}
\bibitem{almira_antonio} \textbf{J. M. Almira, A. J.  L\'{o}pez-Moreno, } On solutions of the Fr\'{e}chet functional equation, J. Math. Anal. Appl.  332 (2007), 1119--1133.
%\bibitem{almira_results} \textbf{J. M. Almira, } A note on classical and $p$-adic Fr\'{e}chet functional equations with restrictions, Results. Math. 63 (2013) 649-656.   
%J. Math. Anal. Appl.  332 (2012), 1119-1133.
% \bibitem{czerwik} \textbf{S. Czerwik, } \textit{Functional equations and inequalities in several variables,} World Scientific, 2002.
\bibitem{Dj} \textbf{D. Z. Djokovi\'{c}, } A representation theorem for $(X_1-1)(X_2-1)\cdots(X_n-1)$ and its applications, Ann. Polon. Math. \textbf{22} (1969/1970) 189-198.
\bibitem{frechet} \textbf{M. Fr\'{e}chet, }   Une definition fonctionelle des polynomes,
Nouv. Ann. 9 (1909), 145-162.
\bibitem{ger1}  \textbf{R. Ger}, On some properties of polynomial functions,
Ann. Pol. Math. \textbf{25 }(1971) 195-203.
\bibitem{ger} \textbf{R. Ger, } On extensions of polynomial functions, Results in
Mathematics 26 (1994), 281-289.
\bibitem{gouvea}  \textbf{F. G. Gouv\^{e}a, } \textit{$p$-adic numbers: an introduction, } UniversityText, Springer, 1997.
\bibitem{HIR} \textbf{D. H. Hyers, G. Isac, T. M. Rassias, } \textit{Stability of functional equations in several variables, } Birkh\"{a}user, 1998. 
\bibitem{JS} \textbf{G. A. Jones, D. Singerman}, \textit{Complex functions. An algebraic and geometric viewpoint}, Cambridge Univ. Press, 1987. 
%\bibitem{jacobson} \textbf{N. Jacobson, } \textit{Basic Algebra I}, Freeman, New York, 1985.
\bibitem{kormes} \textbf{M. Kormes, } On the functional equation $f(x+y)=f(x)+f(y)$, Bulletin of the Amer. Math. Soc. %\textbf{32} (1926) 689-693.
\bibitem{kuczma} \textbf{M. Kuczma}, \textit{An introduction to the theory of functional equations and inequalities, } (Second Edition, Edited by A. Gil\'{a}nyi), Birkh\"{a}user, 2009.
\bibitem{kuczma1} \textbf{M. Kuczma}, On measurable functions with
vanishing differences, Ann. Math. Sil. \textbf{6} (1992) 42-60.
\bibitem{mckiernan}  \textbf{M. A. Mckiernan, }On vanishing n-th ordered
differences and Hamel bases, Ann. Pol. Math. \textbf{19} (1967) 331-336.
\bibitem{montel} \textbf{M. P. Montel, } Sur quelques extensions d'un th\'{e}or\`{e}me de Jacobi, Prace Matematyczno-Fizyczne \textbf{44} (1) (1937) 315-329.
\bibitem{perez} \textbf{C. P\'{e}rez-Garc\'{\i}a, W. H. Schikhof, } \textit{Locally convex spaces over non-Archimedean valued fields, } Cambridge Studies in advanced mathematics 119, Cambridge University Press, New York, 2010.
\bibitem{robert} \textbf{A. M. Robert, } \textit{A course in $p$-adic analysis, } Graduate Texts in Mathematics \textbf{198}, Springer-Verlag, New-York, 2000.
\bibitem{classicalfields} \textbf{H. Salzmann, T. Grundh\"{o}fer, H. H\"{a}hl, R. L\"{o}wen, } \textit{The classical fields. Structural Features of the Real and Rational Numbers, } Encyclopedia of Mathematics and its Applications, \textbf{112}, Cambridge University Press, 2007.
\end{thebibliography}

%%%  ==============================================================

\bigskip

\footnotesize{J. M. Almira

Departamento de Matem\'{a}ticas. Universidad de Ja\'{e}n.

E.P.S. Linares,  C/Alfonso X el Sabio, 28

23700 Linares (Ja\'{e}n) Spain

Email: jmalmira@ujaen.es }

\vspace{1cm}

\footnotesize{Kh. F. Abu-Helaiel

Departamento de Estad\'{\i}stica e Investigaci\'{o}n Operativa. Universidad de Ja\'{e}n.

Campus de Las Lagunillas

23071 Ja\'{e}n, Spain

Email: kabu@ujaen.es
}

\end{document}